\documentclass[11pt]{article}
\usepackage{amssymb,amsmath,amsthm,amsfonts,color,booktabs,subfigure,float}
\usepackage[dvips]{graphicx}
\theoremstyle{plain}
\newtheorem{theorem}{Theorem}[section]
\newtheorem{lemma}{Lemma}[section]

\newtheorem{corollary}{Corollary}[section]

\numberwithin{equation}{section}

\theoremstyle{definition}
\newtheorem{example}{Example}[section]
\newtheorem{remark}{Remark}[section]

\def\dref#1{(\ref{#1})}
\def\dfrac{\displaystyle\frac}
\def\crr{\cr\noalign{\vskip2mm}}
\def\disp{\displaystyle}
\def\H{{\cal H}}
\def\A{{\cal A}}

\def\Z{\mathbb{Z}}
\newcommand{\R}{{\mathbb R}}
\newcommand{\C}{{\mathbb C}}
\newcommand{\N}{{\mathbb N}}

\begin{document}

\title{{\bf Simultaneous Identification of  Damping Coefficient and Initial Value
for PDEs from Boundary Measurement}}
\author{Zhi-Xue Zhao$^{a,b}$ M.K. Banda$^b$,  and  Bao-Zhu Guo$^{c,d}$\footnote{
Corresponding author. Email: bzguo@iss.ac.cn}\\
{\it $^a$School of Mathematical Sciences, }\\{\it Tianjin Normal University, Tianjin 300387, China}\\
{\it $^b$Department of Mathematics and Applied Mathematics,}\\{\it
University of Pretoria,
 Pretoria 0002, South Africa}\\
{\it $^c$Academy of Mathematics and Systems Science, }\\{\it Academia Sinica, Beijing 100190, China, }\\
{\it $^d$School of Computer Science and Applied Mathematics, }\\{\it
University of the Witwatersrand, Johannesburg, South Africa} }
\date{}
\maketitle
\begin{center}
\begin{abstract}

In this paper, the simultaneous identification of damping or
anti-damping coefficient and initial value for some PDEs is
considered. An identification algorithm is proposed based on the
fact that the output of system happens to be decomposed into a
product of an exponential function and a periodic function. The
former contains information of the damping coefficient, while the
latter does not. The convergence and error analysis are also
developed. Three examples, namely an anti-stable wave equation with
boundary anti-damping, the Schr\"odinger equation with internal
anti-damping, and two connected strings with middle joint
anti-damping, are investigated and demonstrated by numerical
simulations to show the effectiveness of the proposed algorithm.

\vspace{0.3cm}


\vspace{0.3cm}


\end{abstract}

\end{center}

\section{Introduction}
\setcounter{equation}{0}

Let $H$ be a Hilbert space with the inner product
$\langle\cdot,\cdot\rangle$ and inner product induced norm
$\|\cdot\|$, and let $Y=\R$ (or $\C$). Consider the dynamic system
in $H$:
\begin{equation}\label{1.1.10}
\left\{
\begin{array}{ll}
\dot{x}(t)=A(q)x(t),\; x(0)=x_0,\\
y(t)=Cx(t)+d(t),
\end{array}
\right.
\end{equation}
where $A(q):D(A(q))\subset H\rightarrow H$ is the system operator
depending on the coefficient $q$, which is assumed to be a generator
of $C_0$-semigroup $\mathcal{T}_q=(T_q(t))_{t\in\R^+}$ on $H$,
$C:H\rightarrow Y$ is the admissible observation operator for
$\mathcal{T}_q$ (\cite{Weiss}), $x_0\in H$ is the initial value, and
$d(t)$ is the external disturbance.

Various PDE control systems with damping mechanism can be formulated
into system \dref{1.1.10}, where $q$ is the damping coefficient. For
a physical system, if the damping is produced by material itself
that dissipates the energy stored in system, then the system keeps
stable. The identification of damping coefficient  has been well
considered for distributed parameter systems  like Kelvin-Voigt
viscoelastic damping coefficient in Euler-Bernoulli beam
investigated in \cite{Banks1}, and a more general theoretical
framework for various classes of parameter estimation problems
presented in \cite{Banks2}. In these works, the inverse problems are
formulated as least square problems and are solved by finite
dimensionalization. For more revelent works, we can refer to the
monograph \cite{Banks4}. Sometimes, however, the source of
instability may arise from the negative damping. One example is the
thermoacoustic instability in duct combustion dynamics and the other
is the stick-slip instability phenomenon in deep oil drilling, see,
for instance, \cite{Bresch1} and the references therein. In such
cases, the negative damping will result in all the eigenvalues
located in the right-half complex plane, and the open-loop plant is
hence ``anti-stable'' (exponentially stable in negative time) and
the $q$ in such kind of system is said to be the anti-damping
coefficient.

A widely investigated problem in recent years is stabilization for
anti-stable systems by imposing feedback  controls. A  breakthrough
on  stabilization for  an anti-stable wave equation was first
reached in \cite{Smyshlyaev1} where a backstepping transformation is
proposed to design the boundary state feedback control. By the
backstepping method, \cite{Guo2} generalizes \cite{Smyshlyaev1} to
two connected anti-stable strings with joint anti-damping. Very
recently, \cite{Guo3,Guo4} investigate  stabilization for
anti-stable wave equation subject to external disturbance coming
through the  boundary input, where the sliding mode control and
active disturbance rejection control technology are employed. It is
worth pointing out that in all aforementioned works, the
anti-damping coefficients are always supposed to be known.

On the other hand, a few stabilization results for anti-stable
systems with unknown anti-damping coefficients are also available.
In \cite{Krstic1}, a full state feedback adaptive control is
designed for an anti-stable wave equation. By converting the wave
equation into a cascade of two delay elements, an adaptive output
feedback control and parameter estimator are designed in
\cite{Bresch1}. Unfortunately, no convergence of the parameter
update law is provided in these works.

It can be seen in \cite{Bresch1,Krstic1} that it is the uncertainty
of the anti-damping coefficient that leads to complicated design for
adaptive control and parameter update law. This comes naturally with
the identification of unknown anti-damping coefficient. To the best
of our knowledge, there are few studies on this regard. Our focus in
the present paper is on simultaneous identification for both
anti-damping (or damping) coefficient and initial value for system
\dref{1.1.10}, where the coefficient $q$ is assumed to be in a prior
parameter set $Q=[\underline{q},\overline{q}]$ ($\underline{q}$ or
$\overline{q}$ may be infinity) and the initial value is supposed to
be nonzero.

We proceed as follows. In Section \ref{Sec1}, we propose an
algorithm to identify simultaneously the coefficient and initial
value through the measured observation. The system may not suffer
from disturbance or it may suffer from general bounded disturbance.
In Section \ref{Sec2}, a wave equation with anti-damping term in the
boundary is discussed. A Schr\"odinger equation with internal
anti-damping term is investigated in Section \ref{Sec3}. Section
\ref{Sec4} is devoted to coupled strings with middle joint
anti-damping. In all these sections, numerical simulations are
presented to verify the performance of the proposed algorithms. Some
concluding remarks are presented in Section \ref{Sec5}.

\section{Identification algorithm}\label{Sec1}

Before giving the main results, we introduce the following well
known Ingham's  theorem  \cite{Ingham,Komornik,Young} as Lemma
\ref{Le1.1}.
\begin{lemma}\label{Le1.1}
Assume that the strictly increasing sequence $\{\omega_k\}_{k\in
\Z}$ of real numbers satisfies the gap condtion
\begin{equation}\label{1.12}
\omega_{k+1}-\omega_k\geq \gamma \quad\mbox{for all}\quad k\in\Z,
\end{equation}
for some $\gamma>0$. Then, for all $T>2\pi/\gamma$, there exist two
positive constants $C_1$ and $C_2$,  depending only on $\gamma$ and
$T$,  such that
\begin{equation}\label{1.13}
C_1\sum_{k\in\Z}|a_k|^2\leq \displaystyle\int_0^T
\left|\sum_{k\in\Z}a_k e^{i\omega_k t}\right|^2dt \leq C_2\sum_{k\in
\Z}|a_k|^2,
\end{equation}
for every complex sequence $(a_k)_{k\in\Z}\in \ell^2$, where
\begin{equation}\label{1.14}
C_1=\dfrac{2T}{\pi}\left(1-\dfrac{4\pi^2}{T^2\gamma^2}\right),\;
C_2=\dfrac{8T}{\pi}\left(1+\dfrac{4\pi^2}{T^2\gamma^2}\right).
\end{equation}
\end{lemma}

To begin with, we  suppose that there is no external disturbance in
system \dref{1.1.10}, that is,
\begin{equation}\label{1.1.1}
\left\{
\begin{array}{ll}
\dot{x}(t)=A(q)x(t),\; x(0)=x_0,\\
y(t)=Cx(t).&
\end{array}
\right.
\end{equation}
The succeeding  Theorem \ref{Th1.1} indicates that  identification
of the coefficient $q$ and initial value $x_0$ can be achieved
exactly simultaneously without error for  $A(q)$ with some
structure.

\begin{theorem}\label{Th1.1}
Let $A(q)$ in system \dref{1.1.1} generate a $C_0$-semigroup
$\mathcal{T}_q=(T_q(t))_{t\in\R^+}$ and suppose that $A(q)$ and $C$
satisfy the following conditions:

 (i). $A(q)$ has a compact resolvent and all its eigenvalues
$\{\lambda_n\}_{n\in\N}$ (or $\{\lambda_n\}_{n\in\Z}$)  admit the
following expansion:
\begin{equation}\label{1.1.2}
\lambda_n=f(q)+i\mu_n,\quad \cdots<\mu_n<\mu_{n+1}<\cdots,
\end{equation}
where $f:\;Q\rightarrow \R$ is invertible, $\mu_n$ is independent of
$q$, and there exists an $L>0$ such that
\begin{equation}\label{1.1.3}
\dfrac{\mu_nL}{2\pi}\in\Z  \mbox{ for all } n\in\N.
\end{equation}

(ii). The corresponding eigenvectors $\{\phi_n\}_{n\in\N}$ form a
Riesz basis for $H$.

(iii). There exist two positive numbers $\kappa$ and $K$ such that
$\kappa \leq |\kappa_n|\leq K$ for all $n\in\N$, where
\begin{equation}\label{1.2.5}
\kappa_n:=C\phi_n,\quad n\in\N.
\end{equation}
Then both coefficient $q$ and initial value $x_0$ can be uniquely
determined by the output $y(t),t\in [0,T]$, where $T>2L$. Precisely,
\begin{equation}\label{1.1.5}
q=f^{-1}\left(\dfrac{1}{L}\ln\dfrac{\|y\|_{L^2(T_1,T_2)}}{\|y\|_{L^2(T_1-L,T_2-L)}}\right),
\end{equation}
for any $L<T_1<T_2-L$, and
\begin{equation}\label{1.1.6}
x_0=\dfrac{1}{L}\sum_{n\in\N}\dfrac{1}{\kappa_n}\left(\int_0^L
y(t)e^{-\lambda_n t}dt\right) \phi_n.
\end{equation}
\end{theorem}
\begin{proof}
Since $\{\phi_n\}_{n\in\N}$  forms a Riesz basis for $H$,  there
exists a sequence $\{\psi_n\}_{n\in\N}$ of eigenvectors of
$A(q)^\ast$, which  is biorthogonal to $\{\phi_n\}_{n\in\N}$, that
is, $\langle\phi_n,\psi_m\rangle=\delta_{nm}$. In this way, we can
express the initial value $x_0\in H$ as $x_0=\sum_{n\in\N}\langle
x_0,\psi_n\rangle \phi_n,$ and the solution of  system \dref{1.1.1}
as
\begin{equation}\label{1.1.7}
x(t)=T_q(t)x_0=\sum_{n\in\N}e^{\lambda_n t}\langle x_0,\psi_n\rangle
\phi_n.
\end{equation}
By \dref{1.1.3}, there exists an increasing sequence
$\{K_n\}\subset\Z$ such that
\begin{equation}\label{1.1.9.1}
\mu_n=\dfrac{2\pi K_n}{L},\; n\in\N,
\end{equation}
which implies that $\{\mu_n\}$ satisfies the following gap condition
\begin{equation}\label{1.1.25}
\mu_{n+1}-\mu_n=\dfrac{2\pi (K_{n+1}-K_n)}{L}\geq
\dfrac{2\pi}{L}\triangleq\gamma.
\end{equation}
In addition, by the assumption that $\{\phi_n\}$ forms a Riesz basis
for $H$ and $|C\phi_n|$ is uniformly bounded with respect to $n$, it
follows from Proposition 2 of \cite{Guo1} or Theorem 2 of
\cite{Guo1.1} that $C$ is admissible for $\mathcal{T}_q$. So
generally, we have
\begin{equation}\label{1.1.8}
y(t)=Cx(t)=\disp\sum_{n\in\N}e^{\lambda_n t}\langle
x_0,\psi_n\rangle C\phi_n=e^{f(q)t}\disp\sum_{n\in\N}e^{i\mu_n
t}\langle x_0,\psi_n\rangle C\phi_n\triangleq e^{f(q)t}P_L(t),
\end{equation}
where $P_L(t)=\sum_{n\in\N}e^{i\mu_n t}\langle x_0,\psi_n\rangle
C\phi_n$ is well-defined and it follows from \dref{1.1.9.1} that
$P_L(t)$ is a $Y$-valued function of period $L$. For any
$T_2-L>T_1>L$,
\begin{equation}\label{1.1.9}
\disp\int_{T_1}^{T_2}
\left|y(t)\right|^2dt=e^{2f(q)L}\disp\int_{T_1-L}^{T_2-L}\left|y(t)\right|^2dt,
\end{equation}
that is,
\begin{equation}\label{1.1.9.3}
\|y\|_{L^2(T_1,T_2)}=e^{f(q)L}\|y\|_{L^2(T_1-L,T_2-L)}.
\end{equation}
To obtain \dref{1.1.5}, we need to show that
$\|y\|_{L^2(T_1,T_2)}\neq 0$ for $T_2-T_1>L$. Actually, it follows
from \dref{1.1.8} that
\begin{equation}\label{1.1.24}
\|y\|_{L^2(T_1,T_2)}^2=\disp\int_{T_1}^{T_2}\left|e^{f(q)
t}P_L(t)\right|^2dt\geq
C_3\disp\int_{T_1}^{T_2}\left|\sum_{n\in\N}e^{i\mu_n t}\langle
x_0,\psi_n\rangle C\phi_n\right|^2dt,
\end{equation}
where $C_3=\min\left\{e^{2T_1f(q)},e^{2T_2f(q)}\right\}>0$. By Lemma
\ref{Le1.1} and the gap condition \dref{1.1.25}, it follows that for
$T_2-T_1>\frac{2\pi}{\gamma}=L$,
\begin{equation}\label{1.1.26}
\disp\int_{T_1}^{T_2}\left|\sum_{n\in\N}e^{i\mu_n t}\langle
x_0,\psi_n\rangle C\phi_n\right|^2dt \geq C_1\kappa^2
\disp\sum_{n\in\N}\left|\langle x_0,\psi_n\rangle\right|^2,
\end{equation}
where
\begin{equation*}\label{1.1.27}
C_1=\dfrac{2(T_2-T_1)}{\pi}\left(1-\dfrac{L^2}{(T_2-T_1)^2}\right)>0\;
\mbox{ for } T_2-T_1>L.
\end{equation*}
The inequality \dref{1.1.24} together with \dref{1.1.26} gives
\begin{equation}\label{1.1.28}
\|y\|_{L^2(T_1,T_2)}^2\geq C_1 C_3\kappa^2
\sum_{n\in\N}\left|\langle x_0,\psi_n\rangle\right|^2.
\end{equation}
Notice that $\{\phi_n\}_{n\in\N}$ forms a Riesz basis for $H$ and so
does $\{\psi_n\}_{n\in\N}$ for $H$, there are two positive numbers
$M_1$ and $M_2$ such that
\begin{equation}\label{1.1.29}
M_1\sum_{n\in\N}\left|\langle
x_0,\psi_n\rangle\right|^2\leq\|x_0\|^2\leq
M_2\sum_{n\in\N}\left|\langle x_0,\psi_n\rangle\right|^2.
\end{equation}
Combining \dref{1.1.28} with \dref{1.1.29} yields
\begin{equation}\label{1.3.13}
\|y\|_{L^2(T_1,T_2)}\geq C \|x_0\|>0,
\end{equation}
where $C=\kappa\sqrt{\frac{C_1 C_3}{M_2}}>0$. The identity
\dref{1.1.5} then follows from \dref{1.1.9.3}.

The inequality \dref{1.3.13} means that  system \dref{1.1.1} is
exactly observable for $T_2-T_1>L$. So  the initial value $x_0$ can
be uniquely determined by the output $y(t),t\in [T_1,T_2]$. We show
next how to reconstruct the initial value from the output.

Actually, it follows from \dref{1.1.9.1} that
\begin{equation}\label{1.4.9}
\dfrac{1}{L}\int_0^L e^{i(\mu_m-\mu_n)t}dt=\delta_{nm},
\end{equation}
Hence,
\begin{equation}\label{1.2.9.2}
\disp\int_0^L y(t)e^{-\lambda_n t}dt=\disp\int_0^L\left(
\sum_{m\in\N}e^{i(\mu_m-\mu_n) t}\langle x_0,\psi_m\rangle
C\phi_m\right) dt=\kappa_n L\cdot \langle x_0,\psi_n\rangle,
\end{equation}
Therefore the initial value $x_0$ can be reconstructed by
\begin{equation}\label{1.2.9.3}
x_0=\sum_{n\in\N}\langle x_0,\psi_n\rangle
\phi_n=\dfrac{1}{L}\sum_{n\in\N}\dfrac{1}{\kappa_n}\left(\int_0^L
y(t)e^{-\lambda_n t}dt\right) \phi_n.
\end{equation}
This completes the proof of the theorem.
\end{proof}
\begin{remark}\label{Re1.1}
Clearly, \dref{1.1.5} and \dref{1.1.6} provide an algorithm to
reconstruct $q$ and $x_0$ from the output. It seems that the
condition \dref{1.1.3} is  restrictive but it is  satisfied by some
physical systems discussed in Sections \ref{Sec2}-\ref{Sec4}.
Condition \dref{1.1.3} is only for identification of $q$. For
identification of initial value only, this condition can be removed.
From  numerical standpoint, the function $P_L(t)$ in \dref{1.1.8}
can be approximated by the finite series in \dref{1.1.8} with the
first $N$ terms for sufficiently large $N$. Hence condition
\dref{1.1.3} can be relaxed in numerical algorithm to be

{\bf C1.} There exists an $L$ such that: every $\frac{\mu_n
L}{2\pi}$ is equal to (or close to) some integer for $n\in
\{1,2,\cdots, N\},$ for some sufficiently large $N$.

Obviously, the relaxed condition C1 can still ensure that $P_L(t)$
is close to a function of period $L$. In this case, some points
$\mu_n$ may be very close to each other and the corresponding Riesz
basis property of the family of divided differences of exponentials
$e^{i\mu_n t}$ developed in \cite[Section II.4]{Avdonin1} and
 \cite{Avdonin2,Avdonin3} can be used. For the
third condition, $|C\phi_n|\leq K$ implies that $C$ is admissible
for $\mathcal{T}_q$ which ensures that the output belongs to
$L^2_{loc}(0,\infty;Y)$, and $|C\phi_n|\geq \kappa$ implies that
system \dref{1.1.1} is exactly observable which ensures the unique
determination of the initial value. It is easily seen from
\dref{1.1.9.3} that the coefficient $q$ can always be identified as
long as $\|y\|_{L^2(T_1,T_2)}\neq 0$ for some time interval
$[T_1,T_2]$, which shows that the identifiability of coefficient $q$
does not rely on the exact observability yet approximate
observability.
\end{remark}

\begin{remark}\label{Re1.2}
The condition $T_2-T_1>L$ in Theorem \ref{Th1.1} is only used in
application of Ingham's  inequality in \dref{1.1.26} to ensure that
$\|y\|_{L^2(T_1,T_2)}\neq 0$. In practical applications, however,
this condition is not always necessary. Actually, any $L<T_1<T_2$ is
applicable in \dref{1.1.5} as long as $\|y\|_{L^2(T_1,T_2)}\neq 0$.
Similar remark also applies for Theorem \ref{Th1.3} below.
\end{remark}

\begin{remark}\label{Re1.3} It should be noted that for
identification of damping coefficient in
\cite{Banks1,Banks2,Banks4}, the distributed observations are always
required. In Theorem \ref{Th1.1}, however, we use only  boundary
measurement and our identification algorithm utilizes physics of the
system that the anti-damping coefficient can make the measurement
have an  exponential term in \dref {1.1.8}. We should also point out
that
 identification of the damping
or anti-damping coefficient $q$ in Theorem \ref{Th1.1} does not rely
on the knowledge of the initial value. Actually, after  $q$ being
estimated, there are various methods for   initial value
reconstruction, see, e.g. \cite{Ramdani,Xu} and the references
therein. The idea of the algorithm for   reconstruction of the
initial value here is borrowed from the Riesz basis approach
proposed in \cite{Xu}.
\end{remark}

Now we come to the system with external disturbance which is
inevitable in many situations. Suppose that system \dref{1.1.10} is
corrupted by an unknown general bounded disturbance $d(t)$ in
observation. It should be noted that system \dref{1.1.10} is
supposed to be anti-stable in Theorem \ref{Th1.3} below whereas in
Theorem \ref{Th1.1}, there is no constraint on the stability of
system.

\begin{theorem}\label{Th1.3}
Suppose that system \dref{1.1.10} is anti-stable and all the
conditions in Theorem \ref{Th1.1} are satisfied. If the inverse of
$f(q)$ is continuous and the disturbance $d(t)$ is bounded, {\it
i.e.} $|d(t)|\leq M$ for some $M>0$ and all $t\geq 0$, then for any
$T_2-L>T_1>L$,
\begin{equation}\label{1.1.19}
\lim_{T_1\rightarrow +\infty}q_{T_1}=q,\;\; \lim_{T_1\rightarrow
+\infty}\|\hat{x}_{0T_1}-x_0\|=0,
\end{equation}
where
\begin{equation}\label{1.1.20}
q_{T_1}=f^{-1}\left(\dfrac{1}{L}\ln\dfrac{\|y\|_{L^2(T_1,T_2)}}{\|y\|_{L^2(T_1-L,T_2-L)}}\right),
\end{equation}
and
\begin{equation}\label{1.3.20}
\hat{x}_{0T_1}=\dfrac{1}{L}\sum_{n\in\N}\dfrac{1}{\kappa_n}\left(\int_{T_1}^{T_1+L}
y(t) e^{-\lambda_n t}dt\right) \phi_n,\; T_1\geq 0.
\end{equation}
Moreover, for sufficiently large $T_1$, the errors
$|f(q_{T_1})-f(q)|$ and $\|\hat{x}_{0T_1}-x_0\|$ satisfy
\begin{equation}\label{1.1.20.1}
|f(q_{T_1})-f(q)|<
\dfrac{4}{L}\dfrac{M\sqrt{T_2-T_1}}{\|y\|_{L^2(T_1-L,T_2-L)}-M\sqrt{T_2-T_1}},
\end{equation}
and
\begin{equation}\label{1.1.21}
\|\hat{x}_{0T_1}-x_0\|\leq \dfrac{CM}{\kappa\sqrt{L}}e^{-f(q)T_1}
\mbox{ for some }  C>0.
\end{equation}
\end{theorem}
\begin{proof}
Introduce
\begin{equation}\label{1.1.22}
y_e(t)=CT_q(t)x_0=y(t)-d(t)=e^{f(q)t}P_L(t),
\end{equation}
where $P_L(t)$ is defined in \dref{1.1.8}. We first show  that
\begin{equation}\label{1.1.23}
\lim_{T_1\rightarrow +\infty}\|y_e\|_{L^2(T_1,T_2)}=+\infty.
\end{equation}
Since system \dref{1.1.10} is anti-stable, the real part of the
eigenvalues $f(q)>0$. It then follows from \dref{1.1.22} that
\begin{equation}\label{1.1.24.1}
\|y_e\|_{L^2(T_1,T_2)}^2=\disp\int_{T_1}^{T_2}\left|e^{f(q)
t}P_L(t)\right|^2dt\geq e^{2f(q)
T_1}\disp\int_{T_1}^{T_2}\left|\sum_{n\in\N}e^{i\mu_n t}\langle
x_0,\psi_n\rangle C\phi_n\right|^2dt.
\end{equation}
Using the same arguments as \dref{1.1.24}-\dref{1.3.13} in the proof
of Theorem \ref{Th1.1}, we have
\begin{equation}\label{1.2.13}
\|y_e\|_{L^2(T_1,T_2)}\geq C e^{f(q) T_1}\|x_0\|,
\end{equation}
where $C=\kappa\sqrt{\frac{C_1}{M_2}}>0$. Since $f(q)>0,\; x_0\neq
0,$  \dref{1.1.23} holds. Therefore for sufficiently large $T_1$,
\begin{equation}\label{1.1.30}
\dfrac{\|y\|_{L^2(T_1,T_2)}}{\|y\|_{L^2(T_1-L,T_2-L)}}=\dfrac{\|y_e+d\|_{L^2(T_1,T_2)}}{\|y_e+d\|_{L^2(T_1-L,T_2-L)}}\leq
\dfrac{\|y_e\|_{L^2(T_1,T_2)}+\|d\|_{L^2(T_1,T_2)}}{\|y_e\|_{L^2(T_1-L,T_2-L)}-\|d\|_{L^2(T_1-L,T_2-L)}}.
\end{equation}
Since $|d(t)|\leq M$,  for any finite time interval $I$,
\begin{equation}\label{1.1.31}
\|d\|_{L^2(I)}=\left(\int_I |d(t)|^2 dt\right)^{\frac{1}{2}}\leq
M\sqrt{|I|},
\end{equation}
where $|I|$ represents the length of the time interval $I$. Hence
\begin{equation}\label{1.1.32}
\dfrac{\|y\|_{L^2(T_1,T_2)}}{\|y\|_{L^2(T_1-L,T_2-L)}}\leq
\dfrac{e^{Lf(q)}+\varepsilon(T_1,T_2)}{1-\varepsilon(T_1,T_2)},
\end{equation}
where
\begin{equation}\label{1.1.34}
\varepsilon(T_1,T_2)=\dfrac{M\sqrt{T_2-T_1}}{\|y_e\|_{L^2(T_1-L,T_2-L)}}.
\end{equation}
Similarly,
\begin{equation}\label{1.1.33}
\dfrac{\|y\|_{L^2(T_1,T_2)}}{\|y\|_{L^2(T_1-L,T_2-L)}}\geq
\dfrac{e^{Lf(q)}-\varepsilon(T_1,T_2)}{1+\varepsilon(T_1,T_2)}.
\end{equation}
It is clear from \dref{1.1.23} and \dref{1.1.34} that
$\lim_{T_1\rightarrow +\infty}\varepsilon (T_1,T_2)=0$. This
together with  \dref{1.1.32} and \dref{1.1.33} gives
\begin{equation}\label{1.1.35}
\lim_{T_1\rightarrow +\infty}
\dfrac{\|y\|_{L^2(T_1,T_2)}}{\|y\|_{L^2(T_1-L,T_2-L)}}=e^{Lf(q)}.
\end{equation}
Since $f^{-1}(q)$ is continuous,
\begin{equation*}\label{1.1.36}
\lim_{T_1\rightarrow
+\infty}q_{T_1}=f^{-1}\left(\dfrac{1}{L}\ln\lim_{T_1\rightarrow
+\infty}\dfrac{\|y\|_{L^2(T_1,T_2)}}{\|y\|_{L^2(T_1-L,T_2-L)}}\right)=q.
\end{equation*}

We next show convergence of the initial value. Similarly with the
arguments \dref{1.4.9}-\dref{1.2.9.3} in the proof of Theorem
\ref{Th1.1},  we have
\begin{equation*}\label{1.3.36}
x_0=\dfrac{1}{L}\sum_{n\in\N}\dfrac{1}{\kappa_n}\left(\int_{T_1}^{T_1+L}
y_e(t) e^{-\lambda_n t}dt\right) \phi_n,\;\forall\; T_1\geq 0.
\end{equation*}
It then follows from \dref{1.3.20} that for arbitrary $T_1\geq 0$,
\begin{equation}\label{1.3.37}
\hat{x}_{0T_1}-x_0=\dfrac{1}{L}\sum_{n\in\N}\dfrac{1}{\kappa_n}\left(\int_{T_1}^{T_1+L}d(t)e^{-\lambda_n
t}dt\right)\phi_n.
\end{equation}
In view of the Riesz basis property of $\{\phi_n\}$, it follows that
\begin{equation}\label{1.1.37}
\begin{array}{ll}
\|\hat{x}_{0T_1}-x_0\|^2&=\left\|\dfrac{1}{L}\sum_{n\in\N}\dfrac{1}{\kappa_n}\left(\int_{T_1}^{T_1+L}d(t)e^{-\lambda_n
t}dt\right)\phi_n\right\|^2\\
&\leq\dfrac{M_2}{L^2\kappa^2}e^{-2f(q)
T_1}\sum_{n\in\N}\left|\int_0^L
\left(d(t+T_1)e^{-f(q)t}\right)e^{-i\mu_n t}dt\right|^2,
\end{array}
\end{equation}
where $M_2>0$ is introduced in \dref{1.1.29}. To estimate the last
series in \dref{1.1.37}, we need the Riesz basis (sequence) property
of the exponential system $\Lambda:=\left\{f_n=e^{i\mu_n
t}\right\}_{n\in\N}$. There are two cases according to the relation
between the sets $\{K_n\}_{n\in\N}$ introduced in \dref{1.1.9.1} and
integers $\Z$:

{\it Case 1:  $\{K_n\}_{n\in\N}=\Z$, that is,
$\Lambda=\left\{e^{i\frac{2n\pi}{L}t}\right\}_{n\in\Z}$}. In this
case, since $\left\{e^{int}\right\}_{n\in\Z}$ forms a Riesz basis
for $L^2[-\pi,\pi]$, $\Lambda$ forms a Riesz basis for
$L^2[-\frac{L}{2},\frac{L}{2}]$.

{\it Case 2:  $\{K_n\}_{n\in\N}\subsetneq\Z$}. In this case, it is
noted  that the exponential system $\left\{e^{i\mu_n
t}\right\}_{n\in\N}$ forms a Riesz sequence in
$L^2[-\frac{L}{2},\frac{L}{2}]$.

In each case above,  by properties  of Riesz basis and Riesz
sequence (see, e.g., \cite[p. 32-35, p.154]{Young}),    there exists
a positive constant $C_4>0$ such that
\begin{equation}\label{1.3.38}
\sum_{n\in\N}|(g,f_n)|^2\leq
C_4\|g\|_{L^2[-\frac{L}{2},\frac{L}{2}]}^2,
\end{equation}
for all $g\in L^2[-\frac{L}{2},\frac{L}{2}]$, where $(\cdot,\cdot)$
 denotes the inner product in $L^2[-\frac{L}{2},\frac{L}{2}]$.

We return to the estimation of $\|\hat{x}_{0T_1}-x_0\|$. By variable
substitution of $t=\frac{L}{2}-s$ in \dref{1.1.37}, together with
\dref{1.3.38},  we have
\begin{equation*}\label{1.3.40}
\begin{array}{ll}
\|\hat{x}_{0T_1}-x_0\|^2&\leq\dfrac{M_2}{L^2\kappa^2}e^{-2f(q)
T_1}\disp\sum_{n\in\N}\left|\int_{-\frac{L}{2}}^{\frac{L}{2}}
\left[d\left(T_1+\frac{L}{2}-s\right)e^{f(q)(s-\frac{L}{2})}e^{-i\mu_n\frac{L}{2}}\right]e^{i\mu_n
s}ds\right|^2\\
&\leq\dfrac{M_2C_4}{L^2\kappa^2}e^{-2f(q)
T_1}\int_{-\frac{L}{2}}^{\frac{L}{2}}
\left|d\left(T_1+\frac{L}{2}-s\right)e^{f(q)(s-\frac{L}{2})}\right|^2ds\\
&\leq\dfrac{M^2M_2C_4}{L\kappa^2}e^{-2f(q) T_1}.
\end{array}
\end{equation*}
Therefore,
\begin{equation}\label{1.3.41}
\|\hat{x}_{0T_1}-x_0\|\leq\sqrt{\dfrac{M_2C_4}{L}}\dfrac{M}{\kappa}e^{-f(q)T_1},
\end{equation}
which implies that $\|\hat{x}_{0T_1}-x_0\|$ will tend to zero as
$T_1\rightarrow +\infty$ for $f(q)>0$. The inequality \dref{1.1.21}
with the positive number $C=\sqrt{M_2C_4}$ is also concluded.

Finally, we estimate $|f(q_{T_1})-f(q)|$. Setting $T_1$ large enough
so that $\varepsilon(T_1,T_2)<1$, it follows from \dref{1.1.20} and
\dref{1.1.32} that
\begin{equation}\label{1.1.41}
Lf(q_{T_1})\leq
\ln\dfrac{e^{Lf(q)}+\varepsilon(T_1,T_2)}{1-\varepsilon(T_1,T_2)}<Lf(q)+\dfrac{2\varepsilon(T_1,T_2)}{1-\varepsilon(T_1,T_2)}.
\end{equation}
Similarly, for $T_1$ large enough so that
$\varepsilon(T_1,T_2)\leq\frac{1}{4}$, it follows from \dref{1.1.20}
and \dref{1.1.33} that
\begin{equation}\label{1.1.42}
Lf(q_{T_1})>Lf(q)-\dfrac{4\varepsilon(T_1,T_2)}{1+\varepsilon(T_1,T_2)}.
\end{equation}

Combining \dref{1.1.41} and \dref{1.1.42}, and setting $T_1$ large
enough so that $\varepsilon(T_1,T_2)\leq\frac{1}{4}$, we have
\begin{equation*}\label{1.1.43}
|f(q_{T_1})-f(q)|<\dfrac{4\varepsilon(T_1,T_2)}{L}.
\end{equation*}
The error estimation \dref{1.1.20.1} comes from the fact
\begin{equation}\label{1.1.44}
\varepsilon(T_1,T_2)\leq\dfrac{M\sqrt{T_2-T_1}}{\|y\|_{L^2(T_1-L,T_2-L)}-M\sqrt{T_2-T_1}}.
\end{equation}
We thus complete the proof of the theorem.
\end{proof}
\begin{remark}\label{Re1.3}
Theorem \ref{Th1.3} shows that when  system \dref{1.1.10} is
anti-stable, then $q_{T_1}$ defined in \dref{1.1.20} can be regarded
as an approximation of the coefficient $q$ when $T_1$ is
sufficiently large. Roughly speaking, the $\varepsilon(T_1,T_2)$
defined in \dref{1.1.34} reflects the ratio of the energy,  in $L^2$
norm, of the disturbance $d(t)$ which is an unwanted signal, with
the energy of the real output signal $y_e(t)$. We may regard
$1/\varepsilon(T_1,T_2)$ as signal-to-noise ratio (SNR) which is
well known in signal analysis. Theorem \ref{Th1.3} indicates that
$q_{T_1}$ defined in \dref{1.1.20} is an approximation of the
coefficient $q$ when SNR is large enough. However, if system
\dref{1.1.10} is stable, {\it i.e.} $f(q)<0$, similar analysis shows
that the output will be exponentially decaying oscillation, which
implies that the unknown disturbance will account for a large
proportion in observation and the SNR can not be too large. In this
case, it is difficult to extract enough useful information from the
corrupted observation as that with large SNR.
\end{remark}
\begin{remark}\label{Re1.4}
The anti-stability assumption in Theorem \ref{Th1.3} is almost
necessary since otherwise, we may have the case of
$y(t)=Cx(t)+d(t)\equiv0$ for which we cannot obtain anything for
identification.
\end{remark}

\begin{remark}\label{Re1.5}
It is well known that the inverse problems are usually ill-posed in
the sense of Hadamard, that is, arbitrarily small error in the
measurement data may lead to large error in solution. Theorem
\ref{Th1.3} shows that if system \dref{1.1.10} is anti-stable, our
algorithm is robust against bounded unknown disturbance in
measurement data. Actually, similar to the analysis in Theorem
\ref{Th1.3}, it can be shown that when system \dref{1.1.10} is not
anti-stable, the algorithm in Theorem \ref{Th1.1} is also
numerically stable in the presence of small perturbations in the
measurement data, as long as the perturbation is relatively small in
comparison to the output. Some numerical simulations validate this
also in  Example \ref{Ex2.1} in Section \ref{Sec2}.
\end{remark}

\section{Application to wave equation}\label{Sec2}

In this section, we apply the algorithm proposed in   previous
section to identification of the anti-damping coefficient and
initial values for a one-dimensional vibrating string equation
described by (\cite{Bresch1,Krstic1})
\begin{equation}\label{1.1}
\left\{\begin{array}{ll} u_{tt}=u_{xx},& 0<x<1,\; t>0,\\
u(0,t)=0,\;\; u_x(1,t)=q u_t(1,t),&t\geq 0,\\
y(t)=u_x(0,t)+d(t),& t\geq 0,\\
u(x,0)=u_0(x),\;\; u_t(x,0)=u_1(x),&0\leq x\leq 1,
\end{array}\right.
\end{equation}
where $x$ denotes the position, $t$ the time, $0<q\neq 1$   the
unknown anti-damping coefficient, $u_0(x)$ and $u_1(x)$   the
unknown initial displacement and initial velocity, respectively, and
$y(t)$ is the boundary measured output corrupted by the disturbance
$d(t)$.

Let $\H=H_E^1(0,1)\times L^2(0,1)$, where $H_E^1(0,1)=\{f\in
H^1(0,1) | f(0)=0\}$,  equipped with the inner product
$\langle\cdot,\cdot\rangle$ and the inner product induced norm
$$\|(f,g)\|^2=\disp\int_0^1\left[|f^\prime (x)|^2+|g(x)|^2\right]dx.$$
Define the system operator $\A: D(\A)(\subset \H)\rightarrow \H$ as
\begin{equation}\label{1.2}
\left\{
\begin{array}{l}
\A(f,g)=(g,f^{\prime\prime}),\\
D(\A)=\big\{(f,g)\in H^2(0,1)\times H_E^1(0,1)\;|\;f(0)=0, f^\prime
(1)=qg(1)\big\},
\end{array}
\right.
\end{equation}
and the observation operator $\mathcal{C}$ from $\H$ to $\C$ as
\begin{equation}\label{1.10}
\mathcal{C}(f,g)=f^\prime (0),\quad (f,g)\in D(\A).
\end{equation}
It is indicated in \cite{Xu} that the operator $\A$  generates a
$C_0$-group on $\H$.

\begin{lemma}\label{Lem1.1}\cite{Xu}
Let $\A$ be defined by \dref{1.2} and let $q\neq 1$. Then the
spectrum of $\A$ consists of all isolated eigenvalues  given by
\begin{equation}\label{1.3}
\lambda_n=\dfrac{1}{2}\ln\dfrac{1+q}{q-1}+in\pi,\; n\in\Z, \;if\;
q>1,
\end{equation}
or
\begin{equation}\label{1.4}
\lambda_n=\dfrac{1}{2}\ln\dfrac{1+q}{1-q}+i\dfrac{2n+1}{2}\pi,\;
n\in\Z, \;if\; 0<q<1,
\end{equation}
and the corresponding eigenfunctions $\Phi_n(x)$ are given by
\begin{equation}\label{1.5}
\Phi_n(x)=\left(\dfrac{\sinh \lambda_n x}{\lambda_n},\sinh \lambda_n
x\right),\quad \forall\;  n\in\Z.
\end{equation}
Moreover, $\left\{\Phi_n(x)\right\}_{n\in\Z}$ forms a Riesz basis
for $\H$.
\end{lemma}

\begin{lemma}\label{Lem1.2}\cite{Xu}
Let $\A$ be defined by \dref{1.2} and let $q\neq 1$. Then the
adjoint operator $\A^\ast$ of $\A$ is given by
\begin{equation}\label{1.6}
\left\{
\begin{array}{l}
\A^\ast(v,h)=-(h,v^{\prime\prime}),\\
D(\A^\ast)=\big\{(v,h)\in H^2(0,1)\times
H^1(0,1)\;|\;v(0)=0,v^\prime (1)+qh(1)=0\big\},
\end{array}
\right.
\end{equation}
and $\sigma (\A^\ast)=\sigma (\A)$. The eigenvector $\Psi_n(x)$ of
$\A^\ast$ corresponding to $\overline{\lambda}_n$ is given by
\begin{equation}\label{1.7}
\Psi_n(x)=\left(\dfrac{\sinh \overline{\lambda}_n
x}{\overline{\lambda}_n},-\sinh \overline{\lambda}_n x\right),\;
\forall\; n\in\Z.
\end{equation}
\end{lemma}

It is easy to verify that for any $n,m\in\Z$,
$\langle\Phi_n,\Psi_m\rangle=\delta_{nm}.$ System \dref{1.1} can be
written as the following evolutionary equation in $\H$:
\begin{equation}\label{1.8}
\dfrac{dX(t)}{dt}=\A X(t),\; t>0,\quad X(0)=(u_0,u_1),
\end{equation}
where $X(t)=(u(\cdot,t),u_t(\cdot,t))$, and the solution of
\dref{1.8} is given by
\begin{equation}\label{1.9}
X(t)=\sum_{n\in\Z}e^{\lambda_n t}\langle X(0),\Psi_n\rangle\Phi_n.
\end{equation}
Thus
\begin{equation}\label{1.11}
y(t)=\sum_{n\in\Z}e^{\lambda_n t}\langle X(0),\Psi_n\rangle+d(t).
\end{equation}

It can be seen from Lemma \ref{Lem1.1} that when $q=1$, the real
part of the eigenvalues is $+\infty$, while for $0<q\neq 1$, the
real part is  finite positive. Hence, we suppose   $1\notin Q$ as
usual (see, e.g., \cite{Bresch1,Krstic1}), where
$Q=[\underline{q},\overline{q}]$ is the prior parameter set.

We take  $q\in Q=(1,+\infty)$ as an example to illustrate how to
apply the algorithms proposed in   previous section to  simultaneous
identification for the anti-damping coefficient $q$ and initial
values. The following Corollaries \ref{Co2.1}-\ref{Co2.3} are the
direct consequences of Theorem \ref{Th1.1} and Theorem \ref{Th1.3},
respectively, by noticing that for system \dref{1.1}, the relevant
function and parameters now are
\begin{equation}\label{1.14}
f(q)=\dfrac{1}{2}\ln\dfrac{q+1}{q-1},\; \mu_n=n\pi,\; L=2,\;
\kappa_n=1.
\end{equation}

\begin{corollary}\label{Co2.1}
Suppose that $d(t)=0$ in system \dref{1.1}. Then  both the
coefficient $q$ and initial values $u_0(x)$ and $u_1(x)$ can be
uniquely determined by the output $y(t),t\in [0,T]$, where $T>4$.
Specifically, $q$ can be recovered exactly  from
\begin{equation}\label{1.12}
q=\dfrac{\|y\|_{L^2(T_1,T_2)}+\|y\|_{L^2(T_1-2,T_2-2)}}{\|y\|_{L^2(T_1,T_2)}-\|y\|_{L^2(T_1-2,T_2-2)}},\;
2\leq T_1<T_2-2,
\end{equation}
and the initial values $u_0(x)$ and $u_1(x)$ can be reconstructed
from
\begin{equation}\label{1.13}
u_0(x)=\dfrac{1}{2}\sum_{n\in\Z}\left(\int_0^2 y(t)e^{-\lambda_n
t}dt\right) \dfrac{\sinh \lambda_n x}{\lambda_n}, \;
u_1(x)=\dfrac{1}{2}\sum_{n\in\Z}\left(\int_0^2 y(t)e^{-\lambda_n
t}dt\right)\sinh \lambda_n x.
\end{equation}
\end{corollary}

Notice that in \dref{1.13}, the observation interval $[0,2]$ is the
minimal time interval for observation to identify the initial values
for any identification algorithm.

\begin{corollary}\label{Co2.3}
Suppose that $q\in Q=(1,+\infty)$ in system \dref{1.1} and the
disturbance is bounded, {\it i.e.} $|d(t)|\leq M$ for some $M>0$ and
all $t\geq 0$. Then for any $T_2-2>T_1\geq 2$,
\begin{equation}\label{1.19}
\lim_{T_1\rightarrow +\infty}q_{T_1}=q,\;\;  \lim_{T_1\rightarrow
+\infty}\|(\hat{u}_{0T_1},\hat{u}_{1T_1})-(u_0,u_1)\|=0,
\end{equation}
where
\begin{equation}\label{1.20}
q_{T_1}=\dfrac{\|y\|_{L^2(T_1,T_2)}+\|y\|_{L^2(T_1-2,T_2-2)}}{\|y\|_{L^2(T_1,T_2)}-\|y\|_{L^2(T_1-2,T_2-2)}},
\end{equation}
and
\begin{equation}\label{1.20.2}
\begin{array}{l}
\hat{u}_{0T_1}(x)=\dfrac{1}{2}\sum_{n\in\Z}\left(\int_{T_1}^{T_1+2}
y(t)e^{-\lambda_n t}dt\right) \dfrac{\sinh \lambda_n x}{\lambda_n},\\
\hat{u}_{1T_1}(x)=\dfrac{1}{2}\sum_{n\in\Z}\left(\int_{T_1}^{T_1+2}
y(t)e^{-\lambda_n t}dt\right)\sinh \lambda_n x.
\end{array}
\end{equation}
\end{corollary}

To end this section, we present some numerical simulations for
system \dref{1.1} to illustrate the performance of the algorithm.

\begin{example}\label{Ex2.1}
{\it The observation with  random noises when system \dref{1.1.10}
is stable.}
\end{example}
A simple spectral analysis together with Theorem \ref{Th1.1} shows
that Corollary \ref{Co2.1} is also valid for $q\in Q=(-\infty,-1)$.
In this example, the damping coefficient $q$ and initial values
$u_0(x),u_1(x)$ are chosen as
\begin{equation}\label{1.21}
q=-3,\quad u_0(x)=-3\sin\pi x,\quad u_1(x)=\pi\cos\pi x.
\end{equation}
In this case, the output can be obtained from \dref{1.11} (with
$d(t)=0$), where the infinite series is approximated by a finite
one, that is, $\{n\in\Z\}$ is replaced by $\left\{n\in\Z\;|-5000\leq
n\leq 5000\right\}$. Some random noises are added to the measurement
data and we use these data to test the algorithm proposed in
Corollary \ref{Co2.1}.

Let $T_1=2,T_2=2.5$. Then  the damping coefficient $q$ can be
recovered from \dref{1.12}, and the initial values $u_0(x)$ and
$u_1(x)$ can be reconstructed from \dref{1.13}. Table \ref{Tab1}
lists the numerical results for the damping coefficients (the second
column in Table \ref{Tab1}) and Figure \ref{Fig-0.1}-\ref{Fig-0.3}
for the initial values in various cases of noise levels. In Table
\ref{Tab1}, the absolute errors of the real damping coefficient and
the recovered ones, and the $L^2$-norm of the differences between
the exact initial values and the reconstructed ones are also shown.

It is worth pointing out that in reconstruction of the initial
values from \dref{1.13}, the infinite series is approximated by a
finite one once again, that is, $\{n\in\Z\}$ is replaced by
$\left\{n\in\Z\;|\;|n|\leq 1000\right\}$, which accounts for the
zero value of the reconstructed initial velocity at the left end.
This is also the reason that the errors of the initial velocity (the
last column in Table \ref{Tab1}) are relatively large even if there
is no random noise in the measured data.

\begin{table}[H]
\caption{Absolute errors with different noise levels} \label{Tab1}
\begin{center}
\begin{tabular}{c c c c c}
\toprule
Noise Level & Recovered $q$ & Errors for $q$ & Errors for $u_0(x)$ & Errors for $u_1(x)$ \\
\midrule
0  &-3.0000&9.3259E-15&1.1744E-08&2.2215E-01\\
1\%&-2.9994&6.2498E-04&1.1618E-03&2.2748E-01\\
3\%&-2.9979&2.0904E-03&3.5604E-03&2.6662E-01\\
\bottomrule
\end{tabular}
\end{center}
\end{table}

\begin{figure}[H]
 \centering
\subfigure[without random noise]{
 \label{Fig-0.1}
 \includegraphics[scale=0.36]{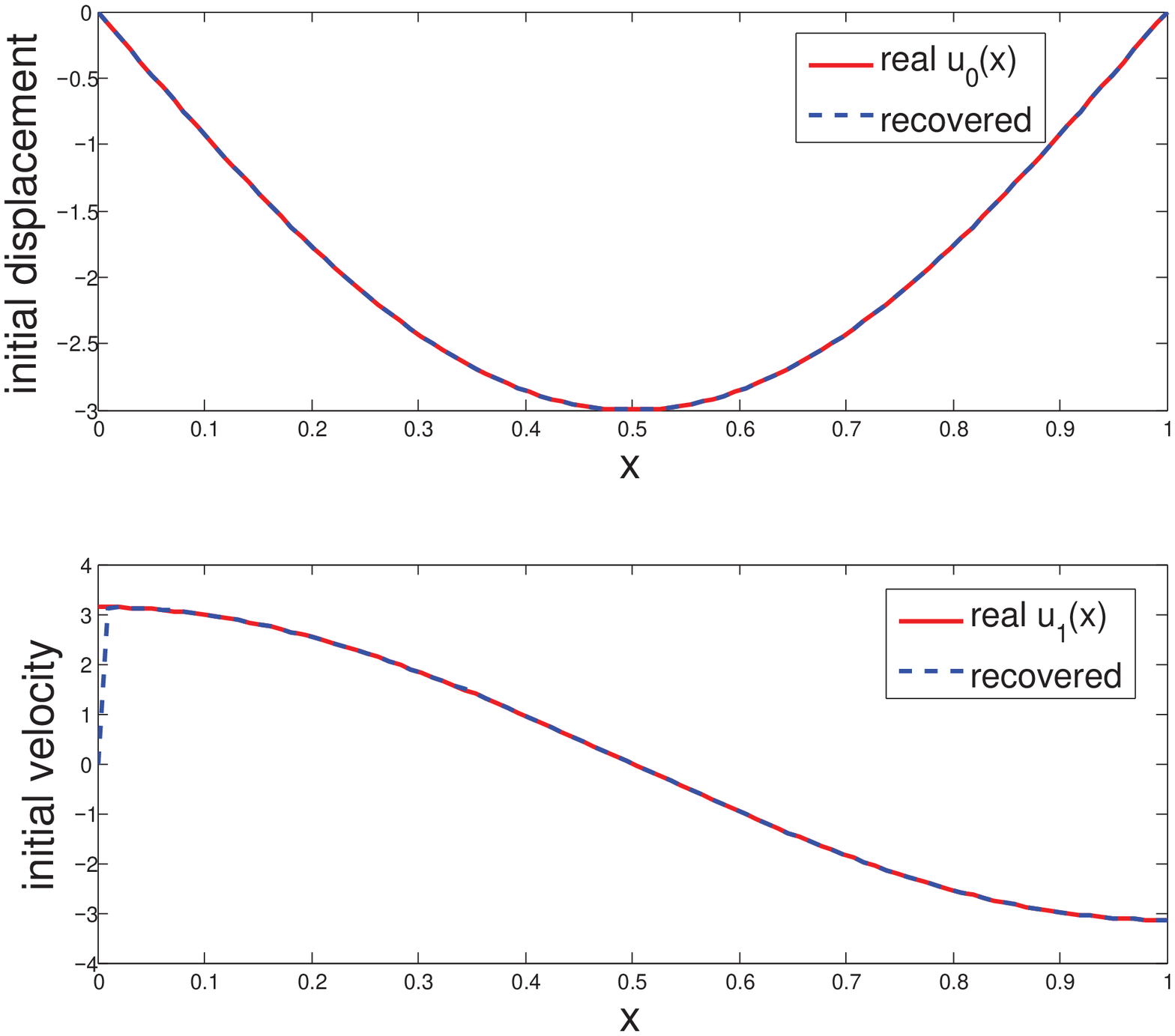}}
 \subfigure[with 1\% random error]{
  \label{Fig-0.2}
 \includegraphics[scale=0.36]{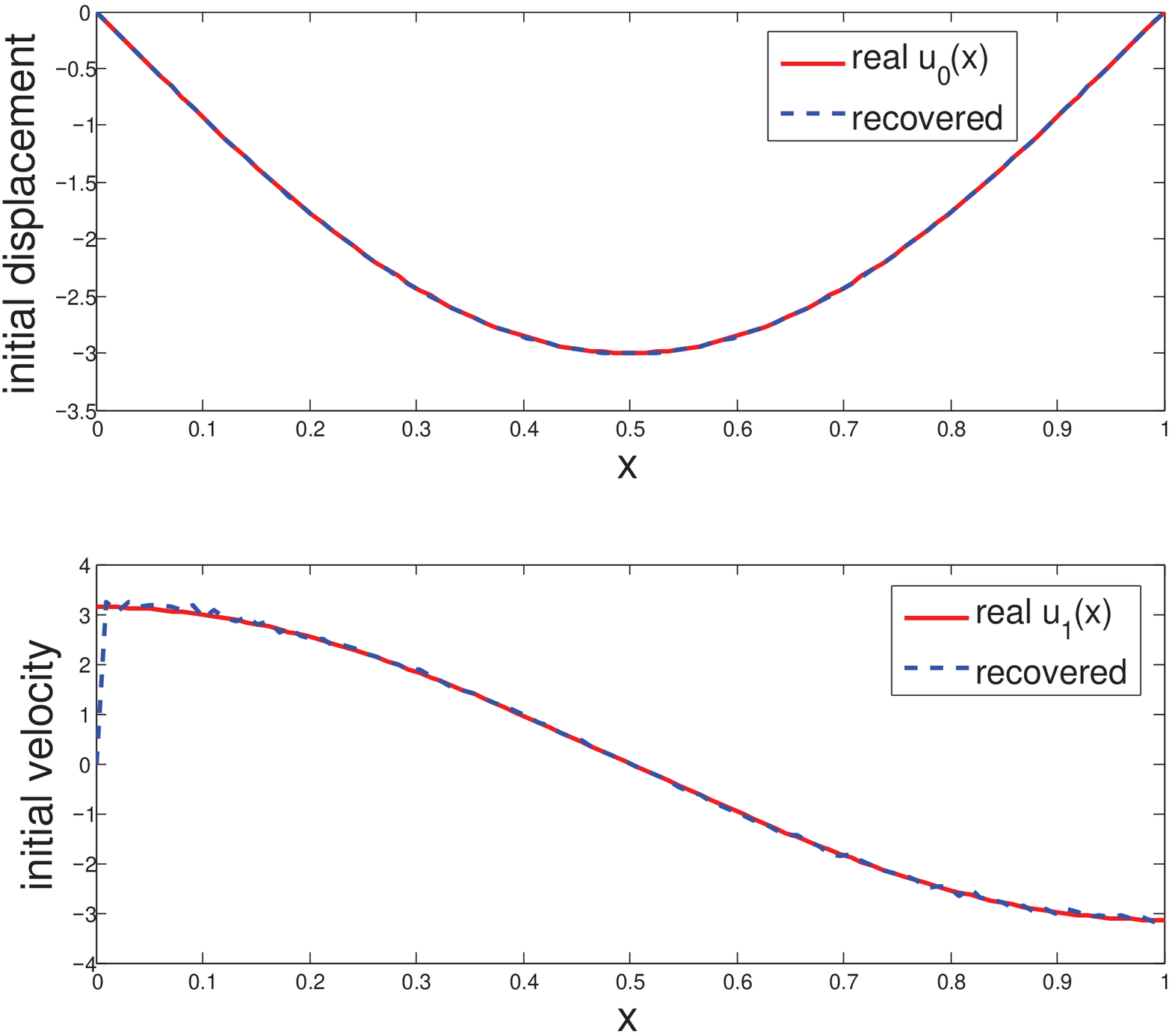}}
 \subfigure[with 3\% random error]{
  \label{Fig-0.3}
 \includegraphics[scale=0.36]{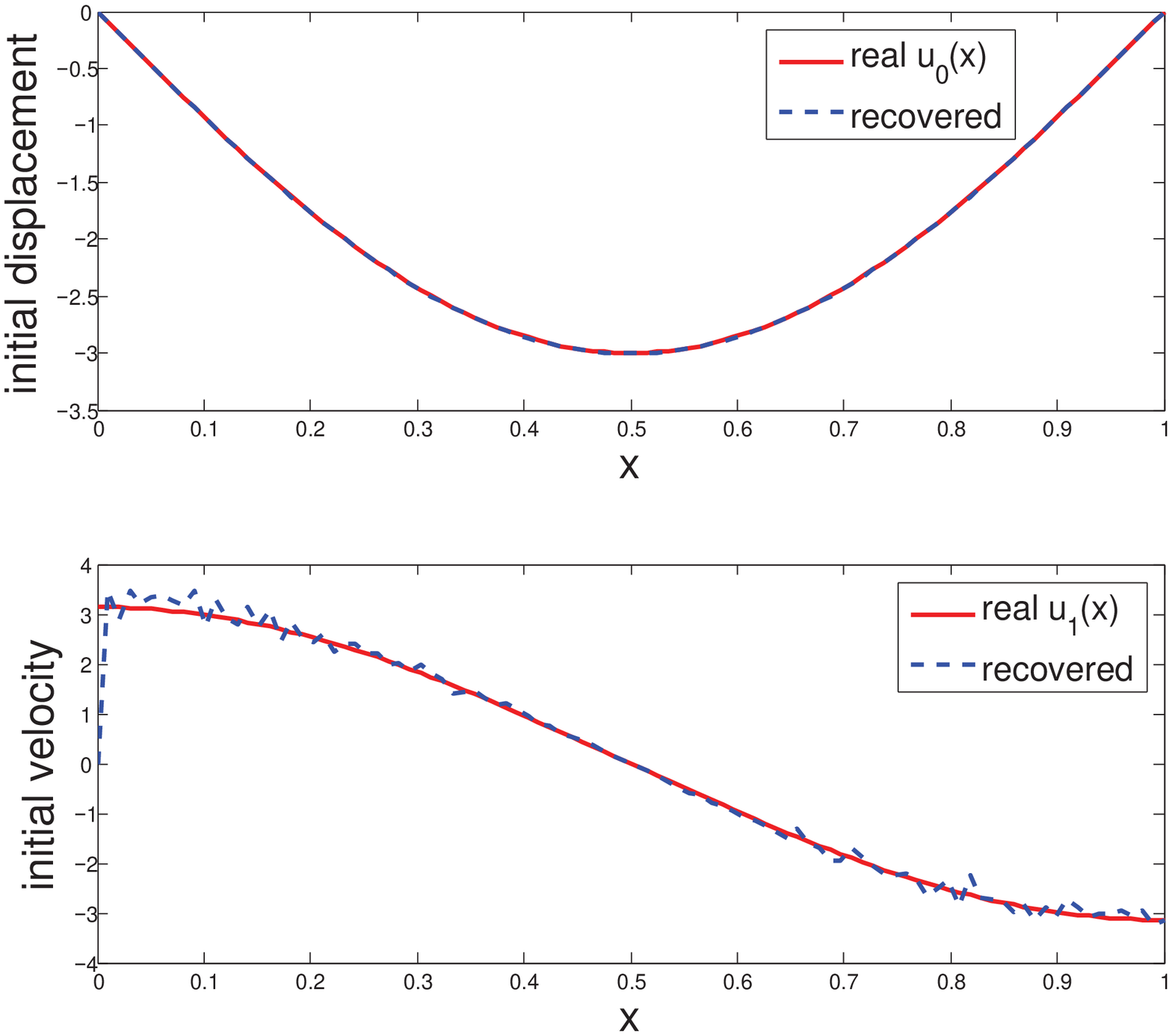}}
\caption{The initial values: initial displacement (upper) and
initial velocity (lower)} \label{Fig-3}
 \end{figure}

\begin{example}\label{Ex2.3} {\it The observation with general bounded disturbance when system
\dref{1.1.10} is anti-stable.} \end{example}

The anti-damping coefficient and initial values are chosen as
\begin{equation}\label{1.21}
q=3,\quad u_0(x)=3\sin\pi x,\quad u_1(x)=\pi\cos\pi x.
\end{equation}
and the observation is corrupted by the bounded disturbance:
\begin{equation}\label{1.24}
d(t)=2\sin\dfrac{1}{1+t}+3\cos 10t.
\end{equation}
The relevant parameters in Corollary \ref{Co2.3} are chosen to be
$T_2=T_1+3$, and let $T_1$ be different values increasing from $2$
to $10$. The corresponding anti-damping coefficients $q_{T_1}$
recovered from \dref{1.20} are depicted in Figure \ref{Fig-2}. It is
seen that $q_{T_1}$ converges to the real value $q=3$ as $T_1$
increases. Setting $T_1=0,\;3,\;7$ in \dref{1.20.2} and
reconstructing the initial values produce results in Figure
\ref{Fig-2} from which we can see that the reconstructed initial
values become closer to the real ones as $T_1$ increases.

\begin{figure}[H]
\centering
\includegraphics[width=3.2in]{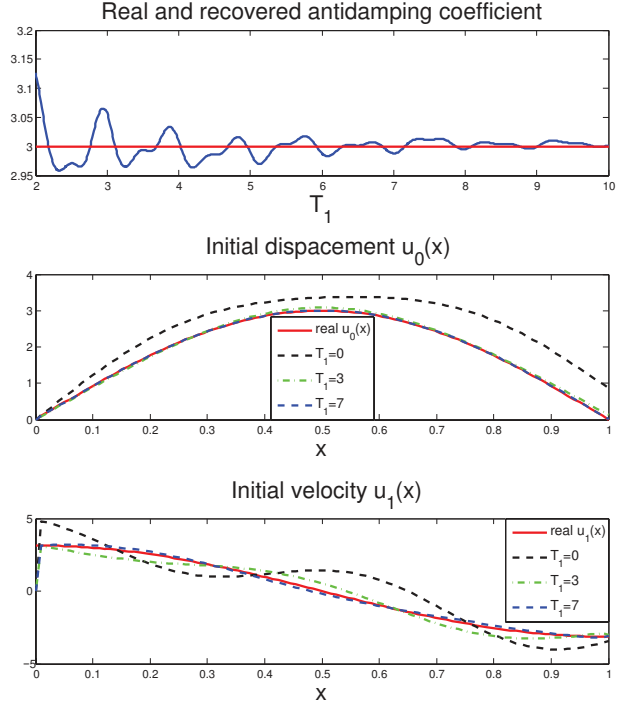}
\caption{anti-damping coefficient $q$ and initial values $u_0,\;
u_1$}\label{Fig-2}
\end{figure}

\section{Application to Schr\"odinger equation}\label{Sec3}

In this section, we consider a quantum system described by the
following Schr\"odinger equation:
\begin{equation}\label{3.1}
\left\{\begin{array}{ll} u_{t}=-i u_{xx}+qu,&0<x<1,\; t>0,\\
u_x(0,t)=0,\;\; u(1,t)=0,&t\geq 0,\\
y(t)=u(0,t)+d(t),& t\geq 0,\\
u(x,0)=u_0(x),&0\leq x\leq 1.
\end{array}\right.
\end{equation}
where $u(x,t)$ is the complex-valued state, $i$ is the imaginary
unit, and the potential $q>0$ and $u_0(x)$ are the unknown
anti-damping coefficient and initial value, respectively.

Let $\H=L^2(0,1)$ be equipped with the usual inner product
$\langle\cdot,\cdot\rangle$ and the inner product induced norm
$\|\cdot\|.$ Introduce the operator $\A$ defined by
\begin{equation}\label{3.2}
\left\{
\begin{array}{l}
\A\phi=-i\phi^{\prime\prime}+q\phi,\\
D(\A)=\left\{\phi\in H^2(0,1)\;|\; \phi^\prime
(0)=\phi(1)=0\right\}.
\end{array}
\right.
\end{equation}
A straightforward verification shows that such defined $\A$
generates a $C_0$-semigroup on $\H$.
\begin{lemma}\label{Le3.1}\cite{Krstic2}
Let $\A$ be defined by \dref{3.2}. Then the spectrum of $\A$
consists of all isolated eigenvalues  given by
\begin{equation}\label{3.3}
\lambda_n=q+i\left(n-\dfrac{1}{2}\right)^2\pi^2,\quad n\in\N,
\end{equation}
and the corresponding eigenfunctions  $\phi_n(x)$  are given by
\begin{equation}\label{3.4}
\phi_n(x)=\sqrt{2}\cos\left(n-\dfrac{1}{2}\right)\pi x,\quad n\in\N.
\end{equation}
In addition, $\{\phi_n(x)\}_{n\in\N}$ forms an orthonormal basis for
$\H$.
\end{lemma}

System \dref{3.1} can be rewritten as the following evolutionary
equation in $\H$:
\begin{equation}\label{3.5}
\dfrac{dX(t)}{dt}=\A X(t),\; t>0,\quad X(0)=u_0,
\end{equation}
and  the solution of \dref{3.5} is given by
\begin{equation}\label{3.6}
X(t)=\sum_{n\in\N}e^{\lambda_n t}\langle X(0),\phi_n \rangle\phi_n.
\end{equation}
Thus
\begin{equation}\label{3.7}
y(t)=\sqrt{2}\sum_{n\in\N}e^{\lambda_n t}\langle
u_0,\phi_n\rangle+d(t).
\end{equation}

The relevant function and parameters in Theorems
\ref{Th1.1}-\ref{Th1.3} for system \dref{3.1} are
\begin{equation*}\label{3.8}
f(q)=q,\; \mu_n=\left(n-\dfrac{1}{2}\right)^2\pi^2,\;
L=\dfrac{8}{\pi},\; \kappa_n=\sqrt{2}.
\end{equation*}
Parallel to Section \ref{Sec2}, we have two corollaries
corresponding to the exact observation and observation with general
bounded disturbance, respectively, for system \dref{3.1}. Here we
only list the latter one  and the former is omitted.

\begin{corollary}\label{Co3.3}
Suppose that  $q\in Q=(0,+\infty)$ in system \dref{3.1} and the
disturbance is bounded, {\it i.e.} $|d(t)|\leq M$ for some $M>0$ and
all $t\geq 0$. Then for any $T_2-\frac{8}{\pi}>T_1>\frac{8}{\pi}$,
\begin{equation}\label{3.15}
\lim_{T_1\rightarrow +\infty}q_{T_1}=q,\quad \lim_{T_1\rightarrow
+\infty}\|\hat{u}_{0T_1}-u_0\|=0,
\end{equation}
where
\begin{equation}\label{3.16}
q_{T_1}=\dfrac{\pi}{8}\ln\dfrac{\|y\|_{L^2(T_1,T_2)}}{\|y\|_{L^2(T_1-\frac{8}{\pi},T_2-\frac{8}{\pi})}},\;
\frac{8}{\pi}<T_1<T_2-\frac{8}{\pi},
\end{equation}
and
\begin{equation}\label{3.16.2}
\hat{u}_{0T_1}(x)=\dfrac{\pi}{8}\sum_{n\in\Z}\left(\int_{T_1}^{T_1+\frac{8}{\pi}}
y(t)e^{-\lambda_n t}dt\right)\cos\left(n-\dfrac{1}{2}\right)\pi x.
\end{equation}
\end{corollary}

We also give a numerical simulation to test the algorithm proposed
in Corollary \ref{Co3.3} for  system \dref{3.1}, where the
anti-damping coefficient $q$ and initial value $u_0(x)$ are chosen
as
\begin{equation}\label{3.17}
q=0.7,\quad u_0(x)=\sin\pi x+i\cos\pi x,
\end{equation}
and the observation is corrupted by the disturbance
\begin{equation}\label{3.18}
d(t)=2\sin\dfrac{t}{10+t}+3i\cos 20t.
\end{equation}
The  observation can be obtained from \dref{3.7} by a finite series
approximation, that is, $\{n\in\N\}$ is replaced by
$\left\{n\in\N\;|\; n\leq 5000\right\}$. The relevant parameters in
Corollary \ref{Co3.3} are chosen to be $T_2=T_1+1$, and  $T_1$
increasing from $2.55$ to $10$. The corresponding anti-damping
coefficients $q_{T_1}$ recovered from \dref{3.16} are shown in
Figure \ref{Fig-3}. It is obvious that $q_{T_1}$ is convergent to
the real value  $q=0.7$ as $T_1$ increases. Setting $T_1=0,\;3,\;7$,
respectively, in \dref{3.16.2}, the  reconstructed initial values
are shown in Figure \ref{Fig-3} from which it is seen that the
errors between the reconstructed initial values and the real ones
become smaller as $T_1$ increases.

\begin{figure}[H]
\centering
\includegraphics[width=3.2in]{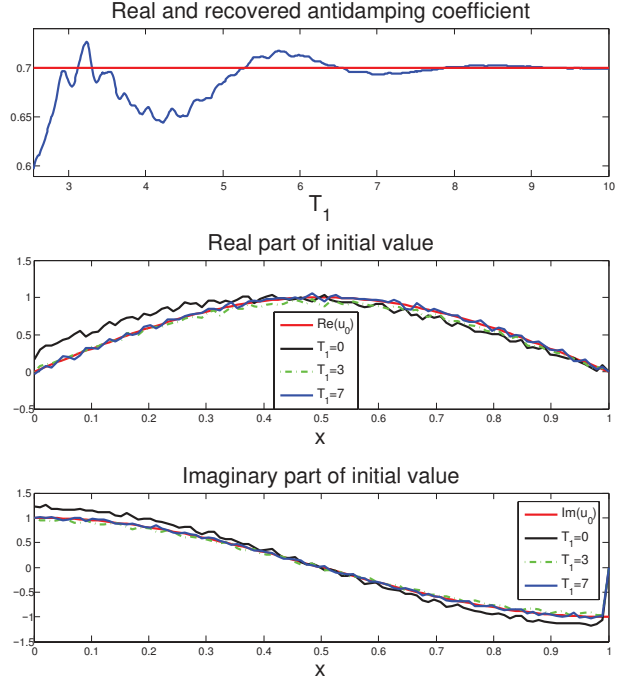}
\caption{coefficient (upper), real part (middle) and imaginary part
(lower) of initial value $u_0(x)$}\label{Fig-3}
\end{figure}

\section{Application to coupled strings equation}\label{Sec4}

In this section, we consider the following two connected anti-stable
strings with joint anti-damping  described by
\begin{equation}\label{4.1}
\left\{\begin{array}{l} u_{tt}(x,t)=u_{xx}(x,t),x\in
(0,\frac{1}{2})\cup (\frac{1}{2},1),\; t>0,\\
u\left(\frac{1}{2}^-,t\right)=u\left(\frac{1}{2}^+,t\right),\;t\geq
0,\\
u_x\left(\frac{1}{2}^-,t\right)-u_x\left(\frac{1}{2}^+,t\right)=qu_t(1,t),\; t\geq 0,\\
u(0,t)=u_x(1,t)=0,\; t\geq 0,\\
u(x,0)=u_0(x),u_t(x,0)=u_1(x),\; 0\leq x\leq 1,\\
y(t)=u_x(0,t)+d(t),\; t\geq 0,
\end{array}\right.
\end{equation}
where $q>0,q\neq 2$ is the unknown anti-damping constant. System
\dref{4.1} models two connected strings with joint vertical force
anti-damping, see \cite{Guo,Guo2,Xu2} for more details.

Let $\H=H_E^1(0,1)\times L^2(0,1)$ be equipped with the inner
product $\langle\cdot,\cdot\rangle$ and its induced norm
$$\|(u,v)\|^2=\disp\int_0^1\left[|u^\prime(x)|^2+|v(x)|^2\right]dx,$$
where $H_E^1(0,1)=\left\{u|\;u\in H^1(0,1),u(0)=0\right\}$. Then
system \dref{4.1} can be rewritten as an evolutionary equation in
$\H$ as follows:
\begin{equation}\label{4.2}
\dfrac{d}{dt}X(t)=\A X(t),
\end{equation}
where $X(t)=(u(\cdot,t),u_t(\cdot,t))\in\H$ and $\A$ is defined by
\begin{equation}\label{4.3}
\A (u,v)=(v(x),u^{\prime\prime}(x)),
\end{equation}
with the domain
\begin{equation}\label{4.4}
\begin{array}{l}
D(\A)=\left\{(u,v)\in H^1(0,1)\times
H_E^1(0,1)\Bigg|\begin{array}{l}
u(0)=u^\prime (1)=0,\;u|_{[0,\frac{1}{2}]}\in H^2(0,\frac{1}{2}),\\
u|_{[\frac{1}{2},1]}\in
H^2(\frac{1}{2},1),\;u^\prime(\frac{1}{2}^-)-u^\prime(\frac{1}{2}^+)=qv(\frac{1}{2}),
\end{array}\right\},
\end{array}
\end{equation}
where $u|_{[a,b]}$ denotes the function $u(x)$ confined to $[a,b]$.

We assume without loss of generality that  the prior parameter set
for $q$ is  $Q=(2,+\infty)$  since the case for $Q=(0,2)$ is very
similar.

\begin{lemma}\label{Lem4.1}\cite{Xu2}
Let $\A$ be defined by \dref{4.3}-\dref{4.4} and $q\in
Q=(2,+\infty)$. Then $\A^{-1}$ is compact on $\H$ and the
eigenvalues of $\A$ are  algebraically simple and separated,   given
by
\begin{equation}\label{4.5}
\lambda_n=\dfrac{1}{2}\ln\dfrac{q+2}{q-2}+in\pi,\; n\in\Z.
\end{equation}
The corresponding eigenfunctions  $\Phi_n(x)$ are  given by
\begin{equation}\label{4.6}
\Phi_n(x)=(\phi_n(x),\lambda_n\phi_n(x)),\; \forall\; n\in\Z,
\end{equation}
where
\begin{equation*}\label{4.7}
\phi_n(x)=\left\{\begin{array}{ll}
\dfrac{\sqrt{2}}{\lambda_n}\cosh\dfrac{\lambda_n}{2}\sinh \lambda_n
x,& 0<x<\dfrac{1}{2},\crr\disp
\dfrac{\sqrt{2}}{\lambda_n}\sinh\dfrac{\lambda_n}{2}\cosh\lambda_n
(1-x),& \dfrac{1}{2}<x<1.
\end{array}\right.
\end{equation*}
and $\{\Phi_n(x)\}_{n\in\Z}$ forms a Riesz basis for $\H$. In
addition, $\A$ generates a $C_0$-semigroup on $\H$.
\end{lemma}

\begin{lemma}\label{Lem4.2}\cite{Xu2}
Let $\A$ be defined by \dref{4.3}-\dref{4.4} and $q\in
Q=(2,+\infty)$. Then the adjoint operator $\A^\ast$ of $\A$ is given
by
\begin{equation}\label{4.8}
\A^\ast(u,v)=-(v,u^{\prime\prime}),
\end{equation}
with the domain
\begin{equation*}\label{4.9}
\begin{array}{l}
D(\A^\ast)=\left\{ (u,v)\in H^1(0,1)\times
H_E^1(0,1)\Bigg|\begin{array}{l}
u(0)=u^\prime (1)=0,\;u|_{[0,\frac{1}{2}]}\in H^2(0,\frac{1}{2}),\\
u|_{[\frac{1}{2},1]}\in
H^2(\frac{1}{2},1),\;u^\prime(\frac{1}{2}^-)-u^\prime(\frac{1}{2}^+)=-qv(\frac{1}{2})
\end{array}
\right\},
\end{array}
\end{equation*}
and $\sigma (\A^\ast)=\sigma (\A)$. The eigenfunctions $\Psi_n(x)$
of $\A^\ast$ corresponding to $\overline{\lambda}_n$ are given by
\begin{equation}\label{4.10}
\Psi_n(x)=\left(\overline{\phi}_n(x),-\overline{\lambda}_n\overline{\phi}_n(x)\right),\quad
\forall\; n\in\Z.
\end{equation}
\end{lemma}
A direct calculation shows that $\{\Psi_n(x)\}$ is biorthogonal to
$\{\Phi_n(x)\}$. Hence, the solution of \dref{4.2} can be expressed
as
\begin{equation}\label{4.12}
X(t)=\sum_{n\in\Z}e^{\lambda_n t}\langle X(0),\Psi_n\rangle\Phi_n.
\end{equation}
Define the observation operator $\mathcal{C}$ from $\H$ to $\C$ to
be
\begin{equation}\label{4.13}
\mathcal{C}(f,g)=f^\prime (0),\quad (f,g)\in D(\A).
\end{equation}
Then
\begin{equation}\label{4.14}
y(t)=\sum_{n\in\Z}e^{\lambda_n t}\langle
X(0),\Psi_n\rangle\mathcal{C}\Phi_n+d(t).
\end{equation}

The succeeding  Corollary \ref{Co4.3} is a  direct consequence of
Theorem \ref{Th1.3} by noticing that for system \dref{4.1}, the
relevant function and parameters now are
\begin{equation*}\label{4.14.2}
f(q)=\dfrac{1}{2}\ln\dfrac{q+2}{q-2},\; \mu_n=n\pi,\; L=2,\;
\kappa_n=\sqrt{2}\cosh\dfrac{\lambda_n}{2},
\end{equation*}
and a simple calculation shows that
\begin{equation*}\label{4.14.3}
\dfrac{\sqrt{2}}{2}\left[\left(\dfrac{q+2}{q-2}\right)^{\frac{1}{4}}-\left(\dfrac{q-2}{q+2}\right)^{\frac{1}{4}}\right]
\leq |\kappa_n| \leq
\dfrac{\sqrt{2}}{2}\left[\left(\dfrac{q+2}{q-2}\right)^{\frac{1}{4}}+\left(\dfrac{q-2}{q+2}\right)^{\frac{1}{4}}\right],
\;\; \forall\; n\in\Z.
\end{equation*}
The corollaries corresponding to Theorem \ref{Th1.1} are  omitted
here.

\begin{corollary}\label{Co4.3}
Suppose that  $q\in Q=(2,+\infty)$ in system \dref{4.1} and the
disturbance is bounded, {\it i.e.} $|d(t)|\leq M$ for some $M>0$ and
all $t\geq 0$. Then for any $T_2-2>T_1\geq 2$,
\begin{equation*}\label{4.21}
\lim_{T_1\rightarrow +\infty}q_{T_1}=q,\; \lim_{T_1\rightarrow
+\infty}\|(\hat{u}_{0T_1},\hat{u}_{1T_1})-(u_0,u_1)\|=0,
\end{equation*}
where
\begin{equation}\label{4.22}
q_{T_1}=\dfrac{2\left(\|y\|_{L^2(T_1,T_2)}+\|y\|_{L^2(T_1-2,T_2-2)}\right)}{\|y\|_{L^2(T_1,T_2)}-\|y\|_{L^2(T_1-2,T_2-2)}},
\end{equation}
and
\begin{equation}\label{4.22.2}
\begin{array}{l}
\hat{u}_{0T_1}(x)=\left\{
\begin{array}{ll}
\disp\dfrac{1}{2}\sum_{n\in\Z}\left(\int_{T_1}^{T_1+2}
y(t)e^{-\lambda_n t}dt\right)\dfrac{\sinh\lambda_n x}{\lambda_n},& 0<x<\dfrac{1}{2},\\
\disp\dfrac{1}{2}\sum_{n\in\Z}\left(\int_{T_1}^{T_1+2}
y(t)e^{-\lambda_n
t}dt\right)\tanh\dfrac{\lambda_n}{2}\dfrac{\cosh\lambda_n
(1-x)}{\lambda_n},& \dfrac{1}{2}<x<1.
\end{array}\right.\crr\disp
\hat{u}_{1T_1}(x)=\left\{
\begin{array}{ll}
\disp\dfrac{1}{2}\sum_{n\in\Z}\left(\int_{T_1}^{T_1+2}
y(t)e^{-\lambda_n t}dt\right)\sinh\lambda_n x,& 0<x<\dfrac{1}{2},\\
\disp\dfrac{1}{2}\sum_{n\in\Z}\left(\int_{T_1}^{T_1+2}
y(t)e^{-\lambda_n t}dt\right)\tanh\dfrac{\lambda_n}{2}\cosh\lambda_n
(1-x),& \dfrac{1}{2}<x<1.
\end{array}\right.
\end{array}
\end{equation}
\end{corollary}

As before, we present some numerical simulations for system
\dref{4.1} to showcase the effectiveness of the algorithm proposed
in Corollary \ref{Co4.3}. The anti-damping coefficient $q$ and
initial values $u_0(x)$ and $u_1(x)$ are chosen to be
\begin{equation}\label{4.23}
q=3,\quad u_0(x)=\sin x,\quad u_1(x)=\cos x,
\end{equation}
and the observation is corrupted by the disturbance
\begin{equation}\label{4.24}
d(t)=\sin\dfrac{t^2}{10+t}+\cos 10t.
\end{equation}
The observation is obtained from \dref{4.14} by a finite series
approximation, that is, $\{n\in\N\}$ is replaced by $\left\{n\in\N
|\; |n|\leq 5000\right\}$. The relevant parameters in Corollary
\ref{Co4.3} are chosen to be $T_2=T_1+1$, and  $T_1$ increases  from
$2$ to $8$. The corresponding anti-damping coefficients $q_{T_1}$
recovered from \dref{4.22} are plotted  in Figure \ref{Fig-4}. It
can be seen that $q_{T_1}$ converges to the real value  $q=3$ as
$T_1$ increases. Let $T_1=0,\;3,\;7$ in \dref{4.22.2}, and the
reconstructed initial values are shown in Figure \ref{Fig-4}, from
which it is seen that the reconstructed initial values become closer
to the real ones as $T_1$ increases.

\begin{figure}[H]
\centering
\includegraphics[width=3.2in]{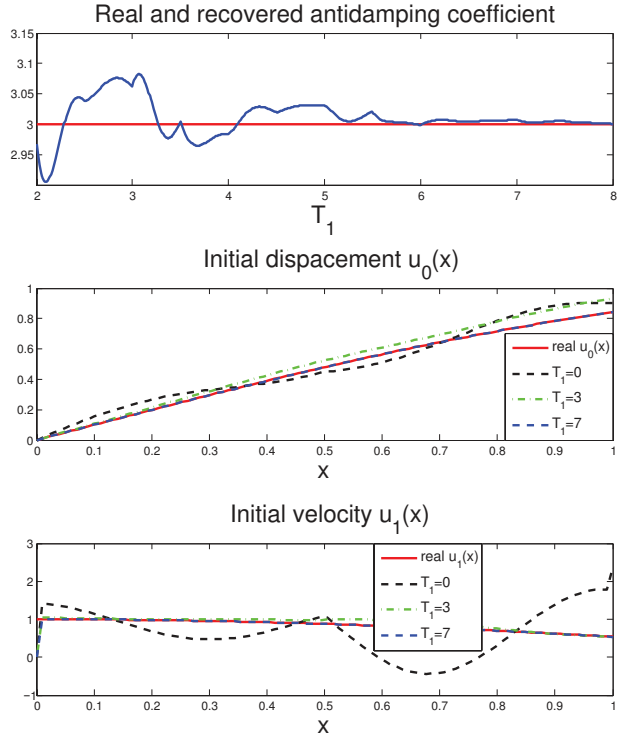}
\caption{anti-damping coefficient $q$ (upper), initial values $u_0$
(middle) and $u_1$ (lower)}\label{Fig-4}
\end{figure}

\section{Concluding remarks}\label{Sec5}

In this paper, we propose an algorithm to reconstruct simultaneously
an anti-damping coefficient and an initial value for some
anti-stable PDEs. When the measured output is exact, the recovered
values are exact whereas if the measured output suffers from bounded
unknown disturbance, the approximated values of the anti-damping
coefficient and initial value can also be obtained. Some numerical
examples are carried out to validate the effectiveness of the
algorithm. It is very promising to apply the algorithm presented
here to stabilization of anti-stable systems with an unknown
anti-damping coefficient.

\section*{Acknowledgements}

This work was supported by the National Natural Science Foundation
of China and the National Research Foundation of South Africa.

\end{document}